\documentclass[12pt]{amsart}
 
\usepackage[T1]{fontenc}
\usepackage{lmodern}          
\usepackage[utf8]{inputenc}   
\usepackage[english]{babel}
 
\usepackage{amssymb}
\usepackage{mathtools}
 
\usepackage[margin=2.5cm]{geometry}
\usepackage{enumitem}
\usepackage{booktabs}         
\usepackage{graphicx}
\usepackage{subfig}
\usepackage{xcolor}
\usepackage{comment}
\usepackage{upquote}
 
\usepackage{amsrefs}
 
\theoremstyle{plain}
\newtheorem{theorem}{Theorem}[section]

\theoremstyle{definition}
\newtheorem{definition}[theorem]{Definition}
\newtheorem{example}[theorem]{Example}
 
\theoremstyle{remark}
\newtheorem{remark}[theorem]{Remark}
 
\numberwithin{equation}{section}
 
\definecolor{linkblue}{rgb}{0.05,0.20,0.55}
\definecolor{citegreen}{rgb}{0.00,0.40,0.15}
\definecolor{urlmagenta}{rgb}{0.50,0.10,0.40}
 
\usepackage{url}
\usepackage{hyperref}
\hypersetup{
    colorlinks=true,
    linkcolor=linkblue,
    citecolor=citegreen,
    urlcolor=urlmagenta,
    pdftitle={Minimal examples of Cohen-Macaulay independence complexes
              that fail to be shellable or vertex decomposable},
    pdfauthor={Adam Van Tuyl}
}

\numberwithin{equation}{section}
\title[Minimal Examples]{Minimal examples of Cohen-Macaulay independence
complexes that fail to be shellable or vertex decomposable}

\author[A. Van Tuyl]{Adam Van Tuyl}
\address[A. Van Tuyl]
{Department of Mathematics and Statistics
McMaster University, Hamilton, ON L8S 4L8, Canada}
\email{vantuyla@mcmaster.ca}

\dedicatory{To Rafael H. Villarreal on the occasion of his 75th birthday}

\begin{document}

\begin{abstract}
We report on an exhaustive census of all 
1,006,700,565 connected graphs 
on $11$ vertices to find the minimal examples of 
graphs 
whose independence complexes are Cohen-Macaulay
but not shellable, or shellable but not vertex decomposable. 
In particular, there are exactly two graphs $G$ where ${\rm Ind}(G)$ is 
Cohen-Macaulay but not shellable, when the characteristic is not two, and exactly six graphs where ${\rm Ind}(G)$ is shellable but not vertex decomposable.
\end{abstract}

\date{\today}

\maketitle

\section{Introduction}
One theme in the study of combinatorial commutative algebra
is to  determine when the independence complex ${\rm Ind}(G)$ 
of a graph $G$ (sometimes called a flag complex) has additional 
topological or combinatorial structure, such as being
{\it Cohen-Macaulay}, {\it shellable}, or 
{\it vertex decomposable} (see 
Section 2 for definitions).  Indeed, Villarreal's paper \cite{V1990},
one of the founding papers in the field,
described conditions on $G$ that force ${\rm Ind}(G)$ to 
be Cohen-Macaulay.
For these three properties,  the following implications always hold:
\begin{equation}\label{eq:implications}
{\rm Ind}(G)~\text{is vertex decomposable}
\Rightarrow {\rm Ind}(G)~\text{is shellable} 
\Rightarrow {\rm Ind}(G)~\text{is Cohen-Macaulay}.
\tag{$\star$}
\end{equation}
In fact, these implications hold for any simplicial complex $\Delta$,
not just independence complexes.  While being vertex 
decomposable and shellable
are combinatorial properties, the Cohen-Macaulay property is defined in terms
of reduced simplicial homology over a field $\mathbb{K}$, or equivalently,
the length of the longest regular sequence in the associated Stanley-Reisner ring. Thus Cohen-Macaulayness also
depends upon the characteristic of $\mathbb{K}$.
If one adds additional hypotheses
on $G$, it can be shown that the converse implications of \eqref{eq:implications} also hold, e.g.,
if $G$ is a bipartite graph \cite{VT2009}, a chordal
graph \cites{DE2009,HHZ2006,W2009}, a well-covered graph with girth
$\geq 5$ \cite{BC2014},  or a very well-covered graph \cite{Metal2011}, among
other families.  

In general, the reverse implications of \eqref{eq:implications} fail.
Moriyama--Takeuchi \cite{MT2003} gave minimal examples of simplicial complexes
on six vertices that are shellable, but not vertex decomposable. However,
these complexes are not the independence complex of any graph;
to the best of our knowledge, the first published example
of a graph $G$ where ${\rm Ind}(G)$ is shellable but not vertex decomposable
is a circulant graph on $16$ vertices found by Earl--Vander Meulen--Van Tuyl \cite{EVMVT2016}.  Using a construction of Moradi--Khosh-Ahang \cite{MKA2016},
Vander Meulen--Van Tuyl \cite{VMVT2017}  leveraged this circulant graph to show that for all $n \geq 16$, there is
a graph on $n$ vertices where the first reverse implication fails. On the 
other hand, Baker--Vander Meulen--Van Tuyl \cite{BVMVT2018} showed via
a computer search that
for all graphs on $n \leq 10$ vertices, 
the implications
of \eqref{eq:implications} can be reversed. 

This leaves the tantalizing question of what are the minimal graphs
for  which \eqref{eq:implications} fails, which must necessarily have
at least 11 vertices. We report on an exhaustive computer
search to fill this gap. 
The following theorem summarizes the findings of our computer
search. 

\begin{theorem}\label{thm:maintheorem}
We have the following classification:
\begin{enumerate}
    \item The six graphs $H_1,H_2,\ldots,H_6$ defined in Section \ref{sec:examples} are the only graphs on 11 vertices with
    ${\rm Ind}(G)$ shellable but not vertex decomposable.
    \item If ${\rm char}(\mathbb{K}) \neq 2$, then the two graphs $G_1$ and $G_2$
    defined in Section \ref{sec:examples} are the only graphs
    on 11 vertices with ${\rm Ind}(G)$ Cohen-Macaulay but not shellable.
    \item If ${\rm char}(\mathbb{K}) =2$, then for all graphs on 11
    vertices, ${\rm Ind}(G)$ is Cohen-Macaulay
    if and only if ${\rm Ind}(G)$ is shellable.
\end{enumerate}
\end{theorem}

We have organized this short note as follows. Section \ref{sec:background}
recalls the relevant definitions, Section \ref{sec:examples} provides
the main examples plus some additional commentary on our examples, and Section \ref{sec:computations}  summarizes 
our computer search.

\section{Background}\label{sec:background}
A {\it simplicial complex} $\Delta$ on a set of vertices $X = \{1,\ldots,n\}$ is a
subset of the power set of $X$ such that: (1) if $F \in \Delta$ and 
$H \subseteq F$, then $H \in \Delta$; and (2) $\{i\} \in \Delta$ for all $i \in X$.  Note that $\emptyset \in \Delta$. An element $F \in \Delta$ is called a {\it face} and the maximal
faces are called the {\it facets}.  If $F_1,\ldots,F_s$ is a complete
list of facets of $\Delta$, then we write $\Delta = \langle F_1,\ldots,F_s\rangle$.  

The {\it dimension} of a face $F \in \Delta$ 
is given by $\dim(F) = |F|-1$; we then define the {\it dimension of $\Delta$}
to be $\dim(\Delta) = \max\{\dim(F) ~|~ F \in \Delta\}$.  A simplicial
complex $\Delta$ is {\it pure} if $\dim(F) = \dim(\Delta)$ for all
of its facets.   If $\dim(\Delta) =d$, then the {\it $f$-vector} of
$\Delta$ is the $(d+2)$-tuple $f(\Delta)= (f_{-1},f_0,\ldots,f_d)$ where
$f_i = |\{F \in \Delta ~|~ \dim(F) = i\}|$. We use the convention that
$\dim(\emptyset) = -1$, and thus $f_{-1} = 1$.  For $\ell = 0,\ldots,d+1$,
set
$h_\ell= \sum_{i=0}^\ell (-1)^{\ell-i}\binom{d+1-i}{\ell-i} f_{i-1}$.
The tuple $h(\Delta)=(h_0,\ldots,h_{d+1})$ is the {\it $h$-vector} of
$\Delta$.

For a simplicial complex $\Delta$ on $X$ and a fixed $i \in X$, the {\it link} of $\{i\}$ in $\Delta$
is the subcomplex
${\rm link}_\Delta(i)   =  \{F \in \Delta ~|~
    F \cup \{i\} \in \Delta,~ F \cap \{i\} = \emptyset\},$
and the {\it deletion} is the simplicial complex
${\rm del}_\Delta(i)  =  \{F \setminus \{i\} ~|~
    F \in \Delta\}$.
    The deletion and link are simplicial complexes
    on $X \setminus \{i\}$.

For a fixed field $\mathbb{K}$, the {\it Stanley-Reisner ideal}
of $\Delta$ on $X = \{1,\ldots,n\}$ is the monomial ideal
$$I_\Delta = \langle x_{i_1}\cdots x_{i_s} ~|~ \{i_1,\ldots,i_s\} \not\in \Delta \rangle \subseteq R= \mathbb{K}[x_1,\ldots,x_n].$$
The quotient ring $R/I_\Delta$ is the {\it Stanley-Reisner ring}
of $\Delta$.
It is known (e.g. \cite[Corollary 6.3.5]{VBook}) that $K$-$\dim(R/I_\Delta) = \dim(\Delta)+1$, 
where $K$-$\dim(R/I_\Delta)$ denotes the Krull dimension 
of $R/I_\Delta$.

We now define the special families of simplicial complexes studied
in this paper.

\begin{definition}  
Let $\Delta$ be a pure simplicial complex on $X$ with
$\dim(\Delta) = d$.
\begin{enumerate}
    \item Then $\Delta$ is {\it vertex decomposable} if either: $(a)$
    $\Delta = \langle F \rangle$ has a single facet, equivalently,
    $\Delta$ is a {\it simplex}, or $(b)$ there exists a vertex $i \in X$
    such that ${\rm link}_\Delta(i)$ and ${\rm del}_\Delta(i)$ are vertex
    decomposable.
    \item Then $\Delta$ is {\it shellable} if there exists an ordering
    $F_1 \prec F_2 \prec \cdots \prec F_s$ of the facets of $\Delta$
    such that for all $1 \leq j < i \leq s$, there exists some $a \in F_i \setminus F_j$ and some $k \in \{1,\ldots,i-1\}$ such that 
    $\{a\} = F_i\setminus F_k$.
    \item Then $\Delta$ is {\it Cohen-Macaulay} if the longest 
    regular sequence in $R/I_\Delta$ has length $d+1$, that is, there
    exist $f_1,\ldots,f_{d+1} \in R$ such that $\overline{f}_1$ is a 
    non-zero divisor
    of $R/I_\Delta$, and for $j=2,\ldots,d+1$, the element 
    $\overline{f}_j$ is a 
    non-zero divisor of $R/(I_\Delta + \langle f_1,\ldots,f_{j-1} \rangle)$,
    where $\overline{f}_k$ denotes the appropriate equivalence class.
\end{enumerate}
\end{definition}

\begin{remark}
    Vertex decomposability was introduced by Provan--Billera \cite{PB1980}.
       The length of the longest regular sequence of a ring is its {\it depth},
    so the Cohen-Macaulay definition is the requirement that
    the depth of $R/I_\Delta$ equal its Krull dimension. 
    An equivalent definition for Cohen-Macaulay simplicial complexes is
    to use Reisner's criterion in terms of the reduced simplicial homology of
    $\Delta$ over a field $\mathbb{K}$ (see \cite[Theorem 5.3.5]{VBook}).  While
    vertex decomposability and shellability only depend on the combinatorics
    of $\Delta$, the Cohen-Macaulay property of $\Delta$ also depends upon
    the characteristic of $\mathbb{K}$.
\end{remark}

We have the following properties, which also justify the implications \eqref{eq:implications} in the introduction.

\begin{theorem}\label{thm:mainproperties}
Let $\Delta$ be a pure simplicial complex on $X$ with 
$\dim(\Delta) =d$.
\begin{enumerate}
    \item \cite[Corollary 2.9]{PB1980} If $\Delta$ is vertex decomposable, then $\Delta$ is shellable.
    \item \cite[Theorem 5.3.18]{VBook} If $\Delta$ is shellable, then $\Delta$ is Cohen-Macaulay (over
    any field $\mathbb{K}$).
    \item  \cite[Theorem 5.4.8]{VBook}
    If $\Delta$ is Cohen-Macaulay, then $h(\Delta)$
    has no negative entries.
\end{enumerate}
\end{theorem}

Moving forward, we focus primarily on the independence complex of a graph.
Given a finite simple graph $G = (V,E)$ with vertex set $V = \{1,\ldots,n\}$
and edge set $E$, we say $W \subseteq V$ is an {\it independent set} if
for every edge $e \in E$, $e \not\subseteq W$.  The complement 
of an independent set is a {\it vertex cover}.  The {\it 
independence complex} of $G$ is
$${\rm Ind}(G) = \{W \subseteq V ~|~ \mbox{$W$ is an independent set}\}.$$
Then ${\rm Ind}(G)$ is a simplicial complex on the vertex set $V = \{1,\ldots,n\}$.  The complex ${\rm Ind}(G)$ is pure  if and only if 
the graph $G$ is {\it well-covered}, that is, every minimal vertex
cover has the same size. The Stanley-Reisner ideal of ${\rm Ind}(G)$ is the 
{\it edge ideal} of $G$, that is, $I(G) = \langle x_ix_j ~|~ \{i,j\} \in E \rangle$ in $R = \mathbb{K}[x_1,\ldots,x_n]$.  


\section{The Examples}\label{sec:examples}

Table \ref{tab:edges} lists the eight 
connected graphs $G_1, G_2,H_1,\ldots,H_6$ on $11$ vertices
found in our exhaustive search.  These 
eight graphs are the minimal examples listed in Theorem \ref{thm:maintheorem} that demonstrate that the 
reverse implications of \eqref{eq:implications}
fail to hold in general.  Graphs $G_1$ and $G_2$ are 
drawn in Figure \ref{fig:CMnotShellable} and graphs
$H_1,\ldots,H_6$ all appear in Figure \ref{fig:ShellablenotVD}.  
We have also included the graph6 representation of each
graph below Table \ref{tab:edges}.

\begin{table}[ht]
\centering
\small
\setlength{\tabcolsep}{4pt}
\begin{tabular}{@{}c*{8}{l}@{}}
\toprule
 & \multicolumn{2}{c}{CM, not shellable} & \multicolumn{6}{c}{Shellable, not VD} \\
\cmidrule(lr){2-3}\cmidrule(l){4-9}
$v$ & $G_1$ & $G_2$ & $H_1$ & $H_2$ & $H_3$ & $H_4$ & $H_5$ & $H_6$ \\
\midrule
 1 & 4,5,7,8 & 4,5,8,9 & 4,6,7,9 & 4,6,7,9 & 4,6,8,9 & 4,6,8,9 & 4,6,8,9 & 4,6,9,10 \\
 2 & 5,6,8,10,11 & 5,6,8,10,11 & 5,8,9,10 & 5,8,9,10 & 5,7,9,10,11 & 5,7,9,10 & 5,7,9,10,11 & 5,7,8,9 \\
 3 & 6,7,9,10 & 6,7,9,10 & 6,7,9,10,11 & 6,7,10,11 & 6,7,10,11 & 6,7,10,11 & 6,7,10,11 & 6,7,8,11 \\
 4 & 8,9,10,11 & 7,8,11 & 7,8,10,11 & 7,8,9 & 7,8,9 & 7,8,9 & 7,8,9 & 8,9,10,11 \\
 5 & 7,9,11 & 9,10,11 & 8,10,11 & 8,10,11 & 8,9,10 & 8,9,10,11 & 8,9,10 & 7,9,10,11 \\
 6 & 8,9 & 8,9,11 & 9,11 & 9,10,11 & 10,11 & 10,11 & 10,11 & 9,10,11 \\
 7 & 10 & 10,11 & 10 & 10,11 & 8,11 & 8,11 & 8,11 & 10 \\
 8 & -- & -- & 11 & 11 & 11 & 11 & -- & 11 \\
 9 & 11 & 11 & -- & -- & -- & -- & 11 & -- \\
10 & 11 & -- & -- & -- & -- & -- & -- & -- \\
\midrule
$|E|$ & 25 & 25 & 24 & 24 & 24 & 24 & 24 & 25 \\
\bottomrule
\end{tabular}
 
\medskip
\begin{tabular}{@{}l l@{\qquad}l l@{\qquad}l l@{\qquad}l l@{}}
\multicolumn{8}{@{}l}{\emph{graph6 codes}}\\[2pt]
$G_1$ & \verb+JCpdUg{[ap_+ & $G_2$ & \verb+JCp`eikYa|?+ & $H_1$ & \verb+JCQdarc]`y?+ & $H_2$ & \verb+JCQdarSZ`]?+ \\
$H_3$ & \verb+JCQbcvWZBM?+ & $H_4$ & \verb+JCQbcvWZ@]?+ & $H_5$ & \verb+JCQbcvWZBL?+ & $H_6$ & \verb+JCQbRb[f`y?+ \\
\end{tabular}
\caption{The eight graphs of Theorem~\ref{thm:maintheorem}. The entry in row $v$ lists the neighbours $w > v$ of vertex $v$, so each edge
$\{v,w\}$ with $v<w$ appears exactly once. Below the table, each graph is also given in graph6 format, with vertex $i$ of the table corresponding to vertex $i-1$ in the graph6 encoding.}\label{tab:edges}
\end{table}

\begin{example}[Cohen-Macaulay but
not shellable]\label{ex:cmbutnotshellable}
Our exhaustive search found 3,247,490 graphs 
on $11$ vertices where ${\rm Ind}(G)$ is Cohen-Macaulay
over $\mathbb{Q}$.  Because all but two of these graphs
are also shellable, 3,247,488 of these graphs
have the property that ${\rm Ind}(G)$ is Cohen-Macaulay over
all fields $\mathbb{K}$.  
The two graphs where ${\rm Ind}(G)$ is Cohen-Macaulay 
but not shellable are the graphs $G_1$ and $G_2$
as given in Table \ref{tab:edges}.
In this case,  ${\rm Ind}(G_i)$ for $i=1,2$ is Cohen-Macaulay when ${\rm char}(\mathbb{K}) \neq 2$, but not shellable. 
The $f$-vector of ${\rm Ind}(G_1)$ and 
${\rm Ind}(G_2)$ is $(1,11,30,20)$.
\end{example}

\begin{figure}
\includegraphics[width=0.35\textwidth]{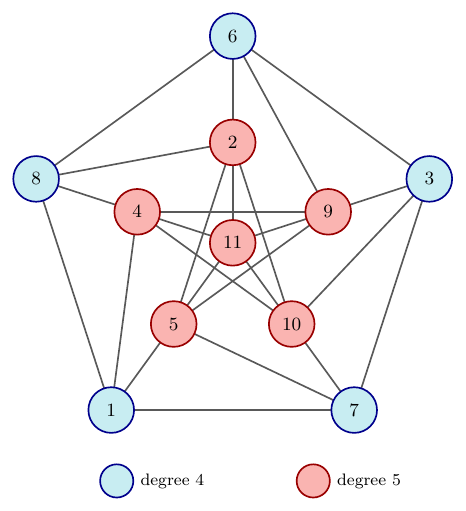}
\includegraphics[width=0.35\textwidth]{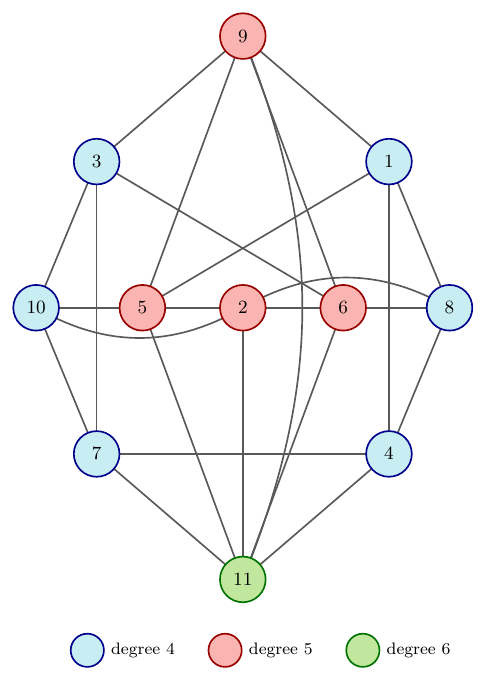}
\caption{The two minimal graphs where ${\rm Ind}(G)$
is Cohen-Macaulay but not shellable. The graph
$G_1$ is on the left and the graph $G_2$ is on the right.}\label{fig:CMnotShellable}
\end{figure}

\begin{proof}[Proof of Theorem \ref{thm:maintheorem} $(2),(3)$]
Recall that the reduced simplicial homology 
of ${\rm Ind}(G)$ only depends
upon the characteristic of the field.   Furthermore,
by the universal coefficient theorem and Reisner's criterion,
if ${\rm Ind}(G)$ is Cohen-Macaulay over some field
$\mathbb{K}$, then ${\rm Ind}(G)$ is Cohen-Macaulay over
$\mathbb{Q}$.  Consequently, if $G$ is a graph on $11$
vertices where ${\rm Ind}(G)$ is 
Cohen-Macaulay over a field $\mathbb{K}$, then 
the graph $G$ is among
the 3,247,490 graphs found in our computer search
which was done over $\mathbb{Q}$.

The independence complexes ${\rm Ind}(G_1)$ and ${\rm Ind}(G_2)$ 
both have the $f$-vector $(1,11,30,20)$, so
$\dim({\rm Ind}(G_i))=2$ for $i=1,2$, and hence
$\dim(R/I(G_i))= 3$ for both graphs.   Graph $G_1$ appears as
the graph $G_4$ in Katzman's paper \cite{K2006}, 
while $G_2$ appears as $G_3$
in the same paper.  Katzman showed that if 
${\rm char}(\mathbb{K}) \neq 2$,
then the projective dimension of $R/I(G_i)$ is 8, and is 9 if 
${\rm char}(\mathbb{K}) = 2$ for $i=1,2$. 
By the Auslander-Buchsbaum Formula,
the depth of $R/I(G_i)$ is then 
$11-8=3$ if  ${\rm char}(\mathbb{K}) \neq 2$, and is $2$ otherwise.  So ${\rm Ind}(G_i)$ is
Cohen-Macaulay if and only if ${\rm char}(\mathbb{K}) \neq 2$.  

On the other hand, ${\rm Ind}(G_i)$ is not shellable for $i=1,2$.
Indeed, if ${\rm Ind}(G_i)$ was shellable for either $i=1,2$, then it would be Cohen-Macaulay in all characteristics by Theorem \ref{thm:mainproperties}. But we just showed this is not the case.

 If $G$ is any other graph on 11 vertices
 with $G \not\in \{G_1,G_2\}$ and ${\rm Ind}(G)$ Cohen-Macaulay
 over some $\mathbb{K}$, and thus over $\mathbb{Q}$, 
 then ${\rm Ind}(G)$ is also shellable.  Indeed,
 if ${\rm Ind}(G)$ is not shellable, then $G,G_1,G_2$ would
 be three graphs whose independence complexes are
 Cohen-Macaulay over
 $\mathbb{Q}$ but not shellable.  But our search only found two
 such graphs.  So $G$ must be among the 3,247,488
 graphs whose independence complexes are shellable.  Thus
 if ${\rm char}(\mathbb{K}) \neq 2$, $G_1$ and $G_2$
 are the only graphs whose independence complexes are Cohen-Macaulay
 but not shellable. However, if ${\rm char}(\mathbb{K}) =2$,
an independence complex of a graph on 11 vertices
is Cohen-Macaulay if and only if it is shellable.
\end{proof}

\begin{remark}
    Both $G_1$ and $G_2$ have previously appeared in the literature,
    but it appears it went unnoticed that
    these graphs also provide minimal examples of graphs
    whose independence complexes are Cohen-Macaulay but not
    shellable.  

    As observed in the above proof,
    our graphs $G_1$ and $G_2$ appear as examples $G_4$ and $G_3$ in the appendix of 
    Katzman's paper \cite{K2006}.    
    In his paper, Katzman showed that the edge ideals of 
    $G_3$ and $G_4$ both have graded Betti numbers
    that can change depending upon the characteristic of the
    field $\mathbb{K}$.  
   
    The graph $G_1$ also appears 
    in Jaramillo--Villarreal's paper
    \cite[Example 5.4]{JV2021} on the ${\tt v}$-number of the
    edge ideal; this graph is used to give the first example
    of a graph whose edge ideal has a
    ${\tt v}$-number that is larger than its
    Castelnuovo-Mumford regularity, providing a counterexample
    to a conjecture of N\'u\~nez-Betancourt--Pitones--Villarreal 
    \cite{NPV2018}.  This graph was further exploited
    by Saha--Van Tuyl \cite[Theorem 3.1]{SVT2026} to give a minimal example
    of a graph whose edge ideal has a
    ${\tt v}$-number greater than the
    degree of its $h$-polynomial.

The  independence complexes of $G_1$ and $G_2$ 
are also related to the minimal clique complexes with torsion found by 
Adamaszek \cite[Theorem 1]{A2014}.   More precisely,
Adamaszek considers
${\rm Cl}(G)$, the clique complex of a graph, where the faces
correspond to the cliques of the graph.  It can be checked
that ${\rm Cl}(G) = {\rm Ind}(\overline{G})$, where $\overline{G}$ denotes
the complement of the graph $G$.  His $K_4$ is the complement of our
$G_1$, and his $K_2 = \overline{G}_2$.

As observed by Adamaszek, the complex ${\rm Cl}(\overline{G}_1)$  
is homeomorphic to the two dimensional 
real projective plane $\mathbb{R}P^2$. Within
the combinatorial commutative algebra community, the same example
is sometimes referred to as {\it Terai's example}.  Indeed, 
the complex ${\rm Ind}(G_1)$, which is attributed to Terai, 
appears as Exercise 6.3.65 in \cite{VBook}
as a triangulation of $\mathbb{R}P^2$ (or Exercise 5.3.31 
of the 2001 edition of \cite{VBook}).
\end{remark}

\begin{example}[Shellable but not vertex decomposable]\label{ex:shellablebutnotvd}
Among all the graphs on $11$ vertices
whose independence complexes are 
shellable, the independence complexes
of the six graphs $H_1,\ldots,H_6$ are the only ones that
fail to
be vertex decomposable.  The $f$-vector 
for ${\rm Ind}(H_i)$ is $(1,11,31,21)$ for $i=1,\ldots,5$,
while the $f$-vector of ${\rm Ind}(H_6)$ is $(1,11,30,20)$.  
The graph $H_6$ is also the only example on $25$ edges,
while the other five graphs have $24$ edges.
\end{example}

\begin{figure}
\centering
\subfloat[Graph $H_1$]{
  \includegraphics[width=50mm]{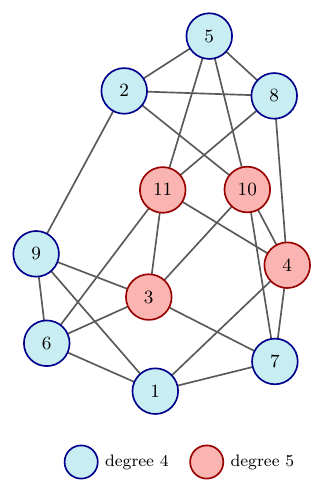}
}
\subfloat[Graph $H_2$]{
  \includegraphics[width=50mm]{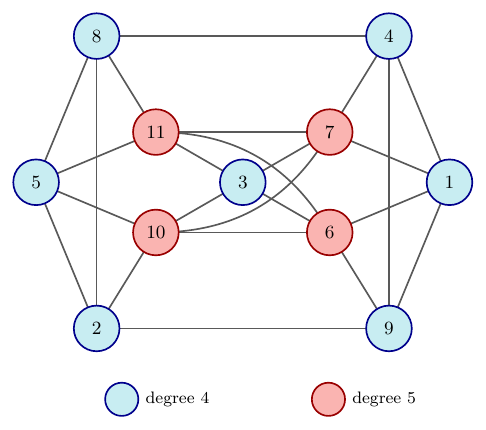}
}
\hspace{0mm}
\subfloat[Graph $H_3$]{
  \includegraphics[width=50mm]{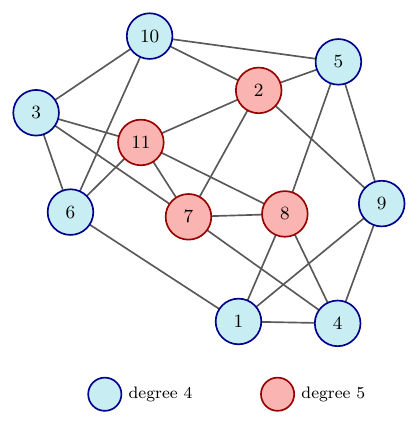}
}
\subfloat[Graph $H_4$]{
  \includegraphics[width=50mm]{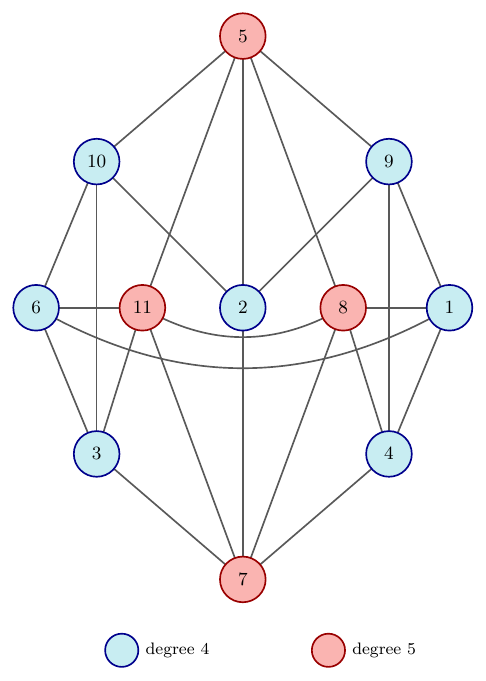}
}
\hspace{0mm}
\subfloat[Graph $H_5$]{   
  \includegraphics[width=50mm]{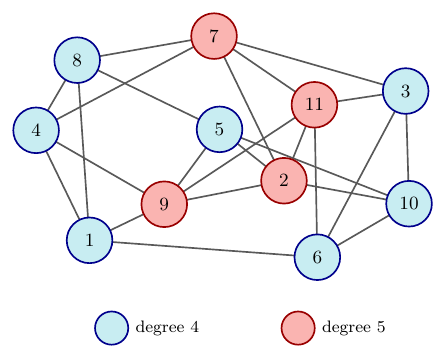}
}
\subfloat[Graph $H_6$]{
  \includegraphics[width=50mm]{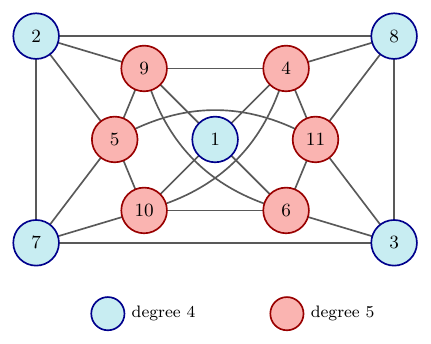}
}
\caption{The six graphs $G$ on $11$ vertices
with ${\rm Ind}(G)$ shellable but not vertex 
decomposable.}\label{fig:ShellablenotVD}
\end{figure}

\begin{proof}[Proof of Theorem \ref{thm:maintheorem} $(1)$]
For each of the six graphs, one can use the {\tt shellingOrder}
command 
in the {\tt SimplicialDecomposability} package \cite{simpdecomp} of {\it Macaulay2} to find a 
shelling order.  As a concrete example, consider the graph $H_1$.
The independence complex ${\rm Ind}(H_1)$ has 
shelling order 
\begin{align*}
&\{9,10,11\},\ \{1,10,11\},\ \{7,9,11\},\ \{2,7,11\},\ \{1,2,11\},\ \{8,9,10\},\ \{6,8,10\},\\
&\{1,8,10\},\ \{7,8,9\},\ \{5,7,9\},\ \{4,5,9\},\ \{6,7,8\},\ \{1,3,8\},\ \{5,6,7\},\ \{2,6,7\},\\
&\{4,5,6\},\ \{2,4,6\},\ \{1,2,3\},\ \{2,3,4\},\ \{3,4,5\},\ \{1,3,5\}.
\end{align*}
On the other hand, if  ${\rm Ind}(H_1)$ was vertex decomposable, 
then there would be  a vertex
$i \in V$ such that ${\rm del}_{{\rm Ind}(H_1)}(i)$ is
vertex decomposable, and thus Cohen-Macaulay by Theorem \ref{thm:mainproperties}.  However, for $i = 1,3, 5$, or $8$, ${\rm del}_{{\rm Ind}(H_1)}(i)$ is not
pure, and thus cannot be Cohen-Macaulay.
For the vertices  $2, 6$ and $9$, the
$h$-vector of  ${\rm del}_{{\rm Ind}(H_1)}(i)$ is $(1,7,8,-1)$,
for vertices $4, 10, 11$, the $h$-vector of 
the deletion is $(1,7,9,-1)$, and
for $i=7$ the $h$-vector is $(1,7,8,-2)$.  So by Theorem \ref{thm:mainproperties} (3), this means none of the deletions are
Cohen-Macaulay, and
hence ${\rm Ind}(H_1)$ is not vertex decomposable.  

Similar calculations show that $H_2,\ldots,H_6$ are shellable but
not vertex decomposable.
\end{proof}

We can use the above examples to construct a graph on
$n \geq 11$ vertices with similar properties.  We
recall the notion of a graph expansion.  For
a graph $G=(V,E)$ on $n$ vertices, fix a tuple ${\bf s}=(s_1,\ldots,s_n) \in
\mathbb{N}^n_{> 0 }$ of positive integers.  The 
{\it ${\bf s}$-expansion of $G$}, denoted $G^{\bf s}$, is the 
graph on the vertex set 
$\{(i,j) ~|~ 1 \leq i \leq n, ~ 1 \leq j \leq s_i\}$
with edge set $$\{\{(i_1,j_1),(i_2,j_2)\} ~|~ \{i_1,i_2\} \in E ~\mbox{or}~ i_1=i_2\}.$$
That is, $G^{{\bf s}}$ is formed by replacing the vertex $i$ with a clique of
size $s_i$, and if $\{i,j\}$ is an edge in the original graph, we join
all vertices in the two corresponding cliques. 

\begin{theorem}\label{thm:expansion}
Let $G$ be a graph and ${\bf s} \in \mathbb{N}^n_{>0}$. 
\begin{enumerate}
    \item \cite[Theorem 2.7]{MKA2016} 
    ${\rm Ind}(G)$ is vertex decomposable if and only if 
    ${\rm Ind}(G^{\bf s})$ is 
    vertex decomposable.
    \item \cite[Theorem 2.12]{MKA2016} If ${\rm Ind}(G)$ is shellable, 
    then ${\rm Ind}(G^{\bf s})$ is shellable.
    \item \cite[Corollary 6.4]{BWW2005} \cite[Theorem 2.4]{RAM2016}
    ${\rm Ind}(G)$ is Cohen-Macaulay over $\mathbb{K}$ if and only if 
    ${\rm Ind}(G^{\bf s})$ is Cohen-Macaulay over $\mathbb{K}$.
\end{enumerate}
\end{theorem}

\begin{remark}
The authors of \cite{MKA2016,RAM2016} study the expansion
of an arbitrary simplicial complex $\Delta$, which is defined in terms
of the facets of $\Delta$.  This definition of an expansion is 
equivalent to the inflation of a simplicial complex found 
in the paper \cite{BWW2005}.  
We have specialized these definitions for ${\rm Ind}(G)$ so that the
expansion/inflation can also be described in terms of the graph $G^{\bf s}$.
Note that \cite{MKA2016} conjectured that Theorem \ref{thm:expansion} (3)
was true, but this result was already known \cite{BWW2005}.
\end{remark}

\begin{theorem}
    For all integers $n \geq 11$
    there exists a graph $G$ on $n$ vertices  such
        that ${\rm Ind}(G)$ is Cohen-Macaulay but not shellable,
        respectively, shellable but not vertex decomposable.
\end{theorem}

\begin{proof}
    For a fixed integer $n \geq 11$, let ${\bf s} = (n-10,1,\dots,1) \in \mathbb{N}^{11}_{>0}$.  Let $G_1$ be the graph of Example \ref{ex:cmbutnotshellable} and
    $H_1$ be the graph of Example \ref{ex:shellablebutnotvd}.  So,
    both $G_1^{\bf s}$ and $H_1^{\bf s}$ are graphs on $n$ vertices.
    By Theorems \ref{thm:maintheorem} and \ref{thm:expansion}, 
    ${\rm Ind}(G_1^{\bf s})$ is Cohen-Macaulay over $\mathbb{K}$ if and only if
    ${\rm Ind}(G_1)$ is Cohen-Macaulay over $\mathbb{K}$ if and only if 
    ${\rm char}(\mathbb{K}) \neq 2$.  Consequently, 
    ${\rm Ind}(G_1^{\bf s})$ is not shellable, because if it was, 
    the complex would be Cohen-Macaulay over all fields.   

    Similarly, by Theorem \ref{thm:expansion}, ${\rm Ind}(H_1^{\bf s})$ is
    shellable, but not vertex decomposable.
\end{proof}

\section{Overview of the computer search}\label{sec:computations} 

Table \ref{table} summarizes our census of the connected graphs
whose independence complexes belong to each of the following families: 
well-covered (pure), Cohen-Macaulay, shellable, and vertex decomposable.
Note that the number of connected graphs
and well-covered graphs on $n$ vertices are given by the sequences
A001349 \cite{oeisA001349} and A222625
\cite{oeisA222625}, respectively.  The first ten rows of the table 
were already computed by Baker--Vander Meulen--Van Tuyl
\cite{BVMVT2018}.  Our contribution is the eleventh row.

\begin{table}[ht]
\centering
\begin{tabular}{@{}rrrrrr@{}}
\toprule
    & Connected & Well-     & Cohen-Macaulay      &           & Vertex        \\
$n$ & graphs    & covered   & (over $\mathbb{Q}$) & Shellable & decomposable  \\
\midrule
 1 &             1 &         1 &         1 &         1 &         1 \\
 2 &             1 &         1 &         1 &         1 &         1 \\
 3 &             2 &         1 &         1 &         1 &         1 \\
 4 &             6 &         3 &         2 &         2 &         2 \\
 5 &            21 &         6 &         5 &         5 &         5 \\
 6 &           112 &        27 &        20 &        20 &        20 \\
 7 &           853 &       108 &        82 &        82 &        82 \\
 8 &        11,117 &       788 &       565 &       565 &       565 \\
 9 &       261,080 &     9,035 &     5,688 &     5,688 &     5,688 \\
10 &    11,716,571 &   196,928 &   102,039 &   102,039 &   102,039 \\
\midrule
11 & 1,006,700,565 & 7,797,877 & 3,247,490 & 3,247,488 & 3,247,482 \\
\bottomrule
\end{tabular}
\caption{Number of connected graphs $G$ on $n$ vertices that are
well-covered (i.e., ${\rm Ind}(G)$ is pure), and for which ${\rm Ind}(G)$
is Cohen-Macaulay over $\mathbb{Q}$, shellable, and vertex decomposable.}\label{table}
\end{table}

To ensure the accuracy of our results, we did our census over all
connected graphs
on $11$ vertices twice. The first time, we used code 
written in {\it Macaulay2} \cite{M2} by the author, and the second
time we used code written in Python with the help of Claude Opus 5.5.  For the
second run, we also computed the number of well-covered graphs in
order to check our computations against the expected number given in
A222625.  In both runs, we checked whether a graph was Cohen-Macaulay over 
the field $\mathbb{Q}$.

The first computer search was carried out between July 2022 and June 2023 while
the author was on sabbatical.  The {\it Macaulay2} code 
made use of the packages
{\tt EdgeIdeals} \cite{edge}, {\tt Nauty} \cite{nauty}, and
{\tt SimplicialDecomposability} \cite{simpdecomp}.   Note that the 
{\tt Nauty} package is a {\it Macaulay2} package  based upon the
nauty program of McKay--Piperno \cite{MP2013}.   The 
computations made use of the supercomputer clusters available to 
members of Digital Research Alliance of Canada (DRAC).
The following work-flow was employed.  Using the {\tt Nauty} package, 
we generated all connected graphs on 
$11$ vertices and $e$ edges where $e \in \{10,\ldots,55\}$.  After
further partitioning this set of graphs into files containing approximately
50,000 graphs, we first
checked which of these graphs were Cohen-Macaulay.  Those that 
were Cohen-Macaulay were then checked for vertex decomposability.
A graph that failed to be vertex decomposable was then further 
checked whether it was shellable or not.  
While this search took advantage of some parallelization, setting up 
the relevant computations and keeping track of the output took up a significant
amount of work time.  All the examples of this paper were found during
the first run in 2022--23.

The second search took place in  September 2026, to validate
that all the examples found in the first search formed a  complete list 
of minimal examples.  Claude Opus 5.5 was
prompted to write code in Python (and the nauty program of McKay--Piperno)
to repeat the census on 11 vertices.  Additionally, the AI model was also prompted to take advantage of parallelization
using the facilities of  DRAC.  
The second run took a day to complete after it was recommended to exploit
Theorem \ref{thm:mainproperties} (3) and make use of memoization when
testing vertex decomposability.  The run returned the exact same eight
graphs described in this paper.

\subsection*{Acknowledgments}
The author thanks Naoki Terai for discussing some of 
the origins of the graph $G_1$ of Example \ref{ex:cmbutnotshellable}. 
Part of this work was made possible by the 
facilities of the Shared Hierarchical Academic 
Research Computing Network 
({\tt SHARCNET:www.sharcnet.ca}) and 
Digital Research Alliance of Canada 
({\tt https://alliancecan.ca/en}).
The author's research is supported by NSERC Discovery Grant 2024--05299. 

\subsection*{Availability of Data} 
The list of all connected graphs on 11 vertices whose independence complexes are Cohen-Macaulay over $\mathbb{Q}$ is available at \url{https://vantuyl.github.io/minimal-cm/}.

\subsection*{Declaration of AI Use}  The author made use of Claude Opus 5.5
to help generate the figures, tables, and improvements for the writing.  As mentioned in
Section \ref{sec:computations}, AI was also used to create code to 
cross-check the results of the first run.  


\end{document}